\documentclass{amsart}

\usepackage{caption}
\usepackage{subcaption}
\usepackage{pifont, xcolor, tikz}
\usepackage{fullpage}

\usepackage{amsmath}

\newtheorem{theorem}{Theorem}
\newtheorem{proposition}[theorem]{Proposition}
\newtheorem{lemma}[theorem]{Lemma}
\newtheorem{conjecture}[theorem]{Conjecture}

\newcommand{\bz}{\ensuremath{\mathbf{z}}}

\newcommand{\oz}{\ensuremath{\overline{z}}}

\newcommand{\sgn}{\operatorname{sgn}}

\usetikzlibrary{positioning}
\usetikzlibrary{decorations.pathreplacing}

\tikzset{partition/.style={fill,circle,inner sep=1pt},
         part/.style={baseline=0,scale=0.5,bend left=45},
         partlabel/.style={below}}
\usetikzlibrary{shapes,arrows}
\tikzstyle{pnt}=[draw,ellipse,fill,inner sep=1pt]

\tikzstyle{opnt}=[ ]
\tikzstyle{pntt}=[draw,ellipse,fill,inner sep=0.5pt]
\tikzstyle{point}=[draw,ellipse,fill,inner sep=2pt]

\makeatletter
\newcommand{\oset}[2]{%
  {\mathop{#2}\limits^{\vbox to \ex@{\kern-\tw@\ex@
   \hbox{\ensuremath{#1}}\vss}}}}
\makeatother

\author{Sophie Burrill}
\author{Stephen Melczer}
\author{Marni Mishna}

\title[New Baxter classes and tableau sequences ending with a row shape]%
{A Baxter class of a different kind, and other bijective results
  using tableau sequences ending with a row shape} 

\address[S. Burrill]{Department of Mathematics, Simon Fraser University, Burnaby BC, Canada, V5A 1S6}
\email{srb7@sfu.ca}

\address[S. Melczer]{Cheriton School of Computer Science, University of Waterloo, Waterloo ON Canada  \&  U. Lyon, CNRS, ENS de Lyon, Inria, UCBL, Laboratoire LIP}
\email{smelczer@uwaterloo.ca}

\address[M. Mishna]{Department of Mathematics, Simon Fraser University, Burnaby BC, Canada, V5A 1S6}
\email{mmishna@sfu.ca}


\begin{document}

  \begin{abstract}
  Tableau sequences of bounded height have been central to the
  analysis of $k$-noncrossing set partitions and matchings. We show
  here that familes of sequences that end with a row shape are
  particularly compelling and lead to some interesting connections.
  First, we prove that hesitating tableaux of height at most two
  ending with a row shape are counted by Baxter numbers. This permits
  us to define three new Baxter classes which, remarkably, do not
  obviously possess the antipodal symmetry of other known Baxter
  classes. We then conjecture that oscillating tableau of height
  bounded by~$k$ ending in a row are in bijection with Young tableaux
  of bounded height~$2k$. We prove this conjecture for $k\leq 8$ by a
  generating function analysis. Many of our proofs are analytic in
  nature, so there are intriguing combinatorial bijections to be
  found.
\end{abstract}

\maketitle

\section{Introduction}
\label{sec:introduction}
The counting sequence for Baxter permutations, whose elements
\[
B_n=\sum_{k=1}^n\frac{\binom{n+1}{k-1}\binom{n+1}{k}\binom{n+1}{k+1}}{\binom{n+1}{1}\binom{n+1}{2}}
\] are known as Baxter
numbers, is a fascinating combinatorial entity, enumerating a
diverse selection of combinatorial classes and resurfacing in many
contexts which do not \emph{a priori\/} appear connected.  The recent
comprehensive survey of Felsner, Fusy, Noy and Orden~\cite{FeFuNoOr11}
finds many structural commonalities among these seemingly diverse
families of objects.

Here we describe three new combinatorial classes which are enumerated
by Baxter numbers. These classes have combinatorial bijections between
them, but {\bf do not} share many of the properties of the other known
Baxter classes aside from a natural Catalan subclass.  Rather, they
connect to the restricted arc diagram family of objects, whose study
was launched by Chen, Deng, Du, Stanley and Yan~\cite{Chetal07}. We
have discovered a bridge between classic Baxter objects and sequences
of tableaux, and consequently walks in Weyl chambers. The main focus
of this article is tableau sequences that end in a (possibly empty)
row shape: as these correspond to a family of lattice walks that end
on a boundary, we are able to deduce enumerative results and find
surprising connections to other known combinatorial
structures. Another consequence of this construction is a new
generating tree description for Baxter numbers.

Our main results are Theorem~\ref{thm:Main} and
Theorem~\ref{thm:Baxter}. They use the integer lattice $W_k=\{(x_k,
\dots, x_1): x_k>\dots, x_1>0, x_i\in\mathbb{N}\}$. All of the other
combinatorial classes referenced in the statement are defined in the
next section.
\begin{theorem}
\label{thm:Main}
The following classes are in bijection for all nonnegative integers
$m, n$ and $k$:
\begin{enumerate}
\item $\mathcal{H}^{(k)}_{n,m}$: the set of hestitating tableaux of length
  $2n$, with maximum height bounded by~$k$, starting from $\emptyset$, ending in a strip of length
  $m$;
\item $\mathcal{L}^{(k)}_{n,m}$: the set of $W_k$-hesitating lattice walks of
  length $2n$, starting at $\delta=(k, k-1, \dots, 1)$ ending at
  $(m+k, k-1, \dots, 1)$;
\item $\Omega^{(k)}_{n,m}$: open partition diagrams of length $n$
  with~$m$ open arcs, but with no~enhanced $(k+1)$-nesting, nor future enhanced~$(k+1)$-nesting.
\end{enumerate}
\end{theorem}
The equivalence (1) $\equiv$ (2) follows almost immediately from the
proof of Theorem~3.6 of~\cite{Chetal07}. We discuss it in
Section~\ref{sec:bijections}. We then prove the equivalence (1)
$\equiv$ (3), by providing an explicit bijection. It is a slight
modification of the bijection $\overline{\phi}$ used in the proof of
Theorem~4.3 of~\cite{Chetal07}.

It is well known that both $\mathcal{L}^{(1)}_{n,0}$ and
$\Omega^{(1)}_{n,0}$ are of cardinalty $C_{n}$, the Catalan number of index $n$.
Furthermore, we have the following theorem, proved in Section~\ref{sec:Baxter}. 
\begin{theorem}
\label{thm:Baxter}
Let $\ell^{(k)}{m}(n)=|\mathcal{L}^{(k)}_{n,m}|$. The number of walks in
$\mathcal{L}^{(2)}_{n,m}$ over all possible values of $m$ is equal to the Baxter number of index
  $n+1$. That is, \[\sum_{m=0}^n \ell^{(2)}_{m}(n)=B_{n+1}.\]
\end{theorem}

The related class of \emph{oscillating} tableaux that end in a row also gives
rise to some interesting results. 
\begin{theorem}
The following classes are in bijection:
\begin{enumerate}
\item Oscillating tableaux of length~$n$, maximum height bounded by~$k$ ending in
  a strip of length $m$;
\item The set of $W_k$-oscillating lattice walks of length $n$ ending
  at $(m+k, k-1, \dots, 1)$;
\item Open matching diagrams of length $n$ with no~$(k+1)$-nesting, nor future~$(k+1)$-nesting, with $m$ open arcs.
\end{enumerate}
\label{thm:Main2}
\end{theorem}
The proof of this theorem is analogous to the proof of the
previous theorem, and we do not present it here. Rather, we are highly
intrigued by a fourth class that also appears to be in bijection.
\begin{conjecture}
  The set of $W_k$-oscillating lattice walks of length $n$ ending
  at the boundary $\{(m+k, k-1, \dots, 1 ): m\geq 0\}$ is in bijection
  with standard Young tableaux of size $n$, with height bounded by~$2k$.
\label{thm:conjecture}
\end{conjecture}
As far as we can tell, this was first conjectured by
Burrill~\cite{Burr14}. Using the lattice path characterization, we
access generating function expressions, and prove the conjecture
for $k\leq 8$. We discuss the strong evidence for Conjecture~\ref{thm:conjecture}
in Section~\ref{sec:Conjecture}.

\section{The Combinatorial Classes}
\label{sec:combclasses}
We now describe the combinatorial classes mentioned in the two main
theorems.
\subsection{Tableaux families}
There are three tableaux families that we consider. 
\begin{description}
\item[oscillating tableau] This is a sequence of Ferrers diagrams such
  that at every stage a box is either added, or deleted. The sequences
  start from an empty shape, and have a specified ending shape. The
  size is the length of the sequence.
\item[standard Young tableau] This is an oscillating tableau where
  boxes are only added and never removed.
\item[hesitating tableau] This is a variant of the oscillating
  tableaux. They are even length sequences of Ferrers diagrams that
  start from an empty tableaux. The sequence is composed of pairs of moves
  of the form: (i) do nothing then add a
  box; (ii) remove a box then do nothing; or (iii) add a box
  then remove a box.
\end{description}
In each case, if no diagram in the sequence is of height $k+1$, we say that the
tableau has height bounded by $k$.  This is how we arrive
to the set $\mathcal{H}^{(k)}_{n,m}$, of hestitating tableaux of
length $2n$, with height bounded by~$k$, ending in the
shape~$(m)$\footnote{Here, $(0)\equiv\emptyset$}.

\subsection{Lattice walks} 
We focus on two different lattice path families. We are interested in walks
in the region\footnote{Remark this model is a reparametrization of the
  region considered in~\cite{Chetal07}.  We slightly shift the notation of
  \cite{BoXi05}.}~$W_k=\{(x_1,x_2, \dots,
x_k): x_i \in \mathbb{Z}, x_1>x_2>\dots>x_k> 0 \}$ starting at the
point $\delta=(k, k-1, \dots, 1)$. Let $e_i$ be the elementary vector
with a 1 at position $i$ and 0 elsewhere. We permit steps that do
nothing, which we call ``stay steps''.
 
The class of \emph{$W_k$-oscillating walks\/} starts at $\delta$ and
takes steps of type $e_i$ or $-e_i$, for $1\leq i\leq k$. We define
$\mathcal{O}^{(k)}_{n,m}$ to be the set of oscillating walks in $W_k$ of
length $n$, ending at the point $me_1+\delta$. 

A \emph{$W_k$-hesitating walk\/} is of even length, and steps occur in the
following pairs: (i) a stay step followed by an $e_i$  step; (ii) a
$-e_i$ step followed by a stay step; (iii) an $e_i$  step follow by $-e_j$ step. 
We focus on~$\mathcal{L}^{(k)}_{n,m}$: the set of hesitating lattice walks of
  length $2n$ in $W_k$, starting at $(k, k-1, \dots, 1)$ and ending at $(m+k, k-1, \dots,1)$.

\subsection{Open arc diagrams}
Matchings and set partitions are two combinatorial classes that have
natural representations using arc diagrams. In the arc diagram 
representation of a set partition of~$\{1, 2, \ldots, n\}$, a row of increasing
vertices is labelled from~$1$ to~$n$. A partition block~$\{a_1, a_2, \ldots, a_j\}$,
ordered~$a_1<a_2<\ldots<a_j$, is represented by the arcs~$(a_1,a_2), (a_2,a_3), \ldots,(a_{j-1}, a_j)$ which are always drawn above the 
row of vertices. The set partition~$\pi=\{\{1,3,7\},\{2,8\},\{4\},\{5,6\}\}$ is depicted 
as an arc diagram in Figure~\ref{fig:expartition}. Matchings are
represented similarly, with each pair contributing an arc. 
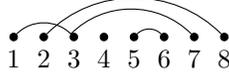
\begin{figure}[h!]
\centering
\begin{tikzpicture}[scale=0.4]
\foreach \x in {1,2,...,8}{
\node[pnt, label=below:{$\x$}](\x) at (\x,0){};}
\draw[bend left=45](1) to (3) to (7);
\draw[bend left=45](2) to (8);
\draw[bend left=45](5) to (6);
\end{tikzpicture}
\caption{The set partition~$\pi=\{1,3,7\},\{2,8\},\{4\},\{5,6\}$}
\label{fig:expartition}
\end{figure} 
A set of $k$ arcs $(i_1, j_1), \dots, (i_k, j_k)$ form a
\emph{$k$-nesting} if $i_1<i_2<\dots <i_k<j_k<\dots<j_2<j_1$. They
form an \emph{enhanced $k$-nesting} if $i_1<i_2<\dots
<i_k\leq j_k<\dots<j_2<j_1$. They form a \emph{$k$-crossing} if
$i_1<i_2<\dots<i_k<j_1<j_2<\dots<j_k$. Figure~\ref{fig:3nest}
illustrates a $3$-nesting, an enhanced $3$-nesting, and a
$3$-crossing. A partition is $k$-nonnesting if it does not contain any
collection of edges that form a $k$-nesting. 

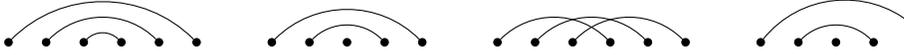
\begin{figure}[h!]
\center
\begin{tikzpicture}[scale=0.5]
\node[pnt] at (0,0)(1){};
\node[pnt] at (1,0)(2){};
\node[pnt] at (2,0)(3){};
\node[pnt] at (3,0)(4){};
\node[pnt] at (4,0)(5){};
\node[pnt] at (5,0)(6){};
\node[pnt] at (7,0)(1c){};
\node[pnt] at (8,0)(2c){};
\node[pnt] at (9,0)(3c){};
\node[pnt] at (10,0)(4c){};
\node[pnt] at (11,0)(5c){};
\node[pnt] at (13,0)(1b){};
\node[pnt] at (14,0)(2b){};
\node[pnt] at (15,0)(3b){};
\node[pnt] at (16,0)(4b){};
\node[pnt] at (17,0)(5b){};
\node[pnt] at (18,0)(6b){};
\node[pnt] at (20,0)(1d){};
\node[pnt] at (21,0)(2d){};
\node[pnt] at (22,0)(3d){};
\node[pnt] at (23,0)(4d){};
\draw(1)  to [bend left=45] (6);
\draw(2)  to [bend left=45] (5);
\draw(3)  to [bend left=45] (4);
\draw(1c)  to [bend left=45] (5c);
\draw(2c)  to [bend left=45] (4c);
\draw(1b)  to [bend left=45] (4b);
\draw(2b)  to [bend left=45] (5b);
\draw(3b)  to [bend left=45] (6b);
\draw(1d)  to [bend left=45] (24,0.5);
\draw(2d)  to [bend left=45] (4d);
\end{tikzpicture}
\caption{A $3$-nesting, an enhanced $3$-nesting, a $3$-crossing, and a
future enhanced $3$-nesting}
\label{fig:3nest}
\end{figure}

Recently, Burrill, Elizalde, Mishna and Yen~\cite{Buetal12}
generalized arc diagrams by permitting semi-arcs: in the diagrams,
each arc must have a left endpoint, but not necessarily a right
endpoint. The notion of $k$-nesting is generalized to
a~\emph{future~$k$-nesting}, which is a pattern with an open arc to
the left of a~$(k-1)$-nesting, and a future enhanced $k$-nesting is
defined similarly. An example is given in
Figure~\ref{fig:3nest}. Burrill~\emph{et al.}~\cite{Buetal12}
conjectured that open partition diagrams with no future enhanced
3-nestings, nor enhanced 3-nestings were counted by Baxter numbers.

\section{Bijections} 
\label{sec:bijections}
We now give the bijections to prove Theorem~\ref{thm:Main}.
Table~\ref{fig:allsizetwo} illustrates the different mappings in the case of $n=2$,
$k=2$. 

\begin{table}[h!]
\center
\begin{tabular}{l||llll}
& {\bf Tableaux} $\mathcal{H}^{(2)}_{2,m}$
& {\bf Walks} $\mathcal{L}^{(2)}_{2,m}$
& {\bf Open partitions} $\Omega^{(2)}_{2,m}$&\\[2mm]
$m=0$
&$\emptyset\, \square\, \emptyset\, \square\, \emptyset $ &
(2,1)-(3,1)-(2,1)-(3,1)-(2,1) & 
\begin{tikzpicture}[scale=0.4]
\foreach \x in {1,2}{\node[pnt] at (\x, 0){};};
\end{tikzpicture}\\

&$\emptyset\, \emptyset\, \square\, \emptyset\, \emptyset $ &
(2,1)-(2,1)-(3,1)-(2,1)-(2,1) &  
\begin{tikzpicture}[scale=0.4]
\foreach \x in {1,2}{\node[pnt] at (\x, 0){};};
\draw[bend left=65](1,0) to (2,0);
\end{tikzpicture}\\[2mm]\hline

&$\emptyset\, \square\, \emptyset\, \emptyset\, \square\,$ &
(2,1)-(3,1)-(2,1)-(2,1)-(3,1) &  
\begin{tikzpicture}[scale=0.4]
\foreach \x in {1,2}{\node[pnt] at (\x, 0){};};
\draw[bend left=45](2,0) to (3,.5);
\end{tikzpicture}\\

$m=1$
&$\emptyset\, \emptyset\, \square\, \square\hspace{-0.5mm}\square \, \square$ & 
(2,1)-(2,1)-(3,1)-(4,1)-(3,1)
&
\begin{tikzpicture}[scale=0.4]
\foreach \x in {1,2}{\node[pnt] at (\x, 0){};};
\draw[bend left=45](1,0) to (2,0);
\draw[bend left=45](2,0) to (3,0.5);
\end{tikzpicture}\\

&$\emptyset\, \emptyset\, \square\, \begin{tikzpicture}[scale=0.21]\draw(0,0) to (1,0) to (1,1) to (0,1) to (0,0);
\draw(1,1) to (1,2) to (0,2) to (0,1); \end{tikzpicture} \,   \square $ &
(2,1)-(2,1)-(3,1)-(3,2)-(3,1) 
&\begin{tikzpicture}[scale=0.4]
\foreach \x in {1,2}{\node[pnt] at (\x, 0){};};
\draw[bend left=20](1,0) to (3,.8);
\end{tikzpicture}\\[2mm]\hline
$m=2$&$\emptyset\, \emptyset\, \square\, \square\,
\square\hspace{-0.5mm}\square $ &
(2,1)-(2,1)-(3,1)-(3,1)-(4,1) &  
\begin{tikzpicture}[scale=0.4]
\foreach \x in {1,2}{\node[pnt] at (\x, 0){};};
\draw[bend left=45](1,0) to (3,1);
\draw[bend left=45](2,0) to (3,0.5);
\end{tikzpicture}\\
&&\\
\end{tabular}
\caption{The six objects of size~$n=2$ for each class in Theorem~\ref{thm:Main} }
\label{fig:allsizetwo}
\end{table}

\subsection{Tableaux to walks}
The bijection between $\mathcal{H}^{(k)}_{n,m}$ and
$\mathcal{L}^{(k)}_{n,m}$ is a straightforward consequence of the
proof of Theorem~3.6 of~\cite{Chetal07}. The authors give the bijection
explicitly for vacillating tableaux, but the results follow directly
for the case of hesitating tableaux. In their bijection, a shape
$\lambda=(\lambda_1, \dots, \lambda_k)$ in a hesitating tableaux
corresponds to a point in position $\delta+\lambda=(\lambda_1+k,
\lambda_2+k-1, \dots, \lambda_k+1)$ in the hestitating walk.

\subsection{Arc diagrams and tableaux}
We modify Chen~\emph{et al.}'s bijection between set partitions and
hesitating tableaux to incorporate the possibility that an edge is not
closed.  In this abstract, we simply describe how to modify Marberg's
presentation in~\cite{Marb13} of the bijection to handle open arcs. We
refer readers to that article, or to Chen~\emph{et al.} to get the
full description.

\begin{figure}
\centering
\begin{tikzpicture}[scale=0.5]
\foreach \x in {1,2,...,5}{
\node[pnt, label=below:{$\x$}](\x) at (\x,0){};}
\node[pnt, label=right:{$8$}](8) at (6,.5){};
\node[pnt, label=right:{$7$}](7) at (6,1.2){};
\node[pnt, label=right:{$6$}](6) at (6,1.9){};
\draw[bend left=45](1) to (3);
\draw[bend left=45](3) to (7);
\draw[bend left=35](2) to (6);
\draw[bend left=25](5) to (8);
\end{tikzpicture}
\caption{Preparing an open diagram for the bijection.}
\label{fig:expartition2}
\end{figure}
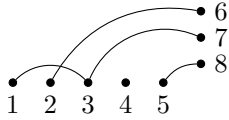

\begin{proposition} The subset of~$\Omega_{n,m}^{(k)}$ comprised of open partition
  diagrams of size~$n$ avoiding both enhanced~$(k+1)$-nestings, and
  future enhanced~$(k+1)$-nestings consisting of diagrams with~$m$ open arcs, is in bijection
  with hesitating tableaux of length~$2n$ of maximum height~$k$,
  ending with a single row of length~$m$.
\end{proposition}

\begin{proof}
We prove this by describing how to process open arcs. We assume the
reader has some familiarity with the original bijection to have a
brief presentation of our bijection.  We read diagrams
from left to right, and insert and delete cells under the same rules as
Marberg~\cite{Marb13}. The insertion depends on the right
endpoints of the arcs. If the diagram has $n$ vertices, we incrementally label the
right endpoints of the open arcs with labels $n+1$ to $n+m$ from left
to right, as seen in Figure~\ref{fig:expartition2}.

Recall how vertices are processed in this bijection. Each vertex
corresponds to two steps; if~$i$ is a left endpoint, first do nothing
and then insert its (possibly artificial) right endpoint into the
tableau; if~$i$ is a right endpoint, delete it from the tableau and
then do nothing; if~$i$ is both a right endpoint of some arc~$(i', i)$
and left endpoint of~$(i, i'')$, first insert~$i''$ into the tableaux
and then delete~$i$. If $i$ is a fixed point, first insert $i$, then
delete $i$. The insertion is done with RSK, and the deletion is simple
cell deletion. 

In this process, the values corresponding to open arcs are inserted in
the tableau in order, and are never deleted. By RSK, they will all be
in the same row. Once the other points are deleted, all that will
remains is a single row of length $m$.

The inverse map of Marberg (and ultimately Chen \emph{et al.})  is similarly
adapted.
\end{proof}

\section{When $k=2$ the complete generating function is Baxter}
\label{sec:Baxter}
\subsection{Baxter permutations}
Baxter numbers comprise entry A001181 of the Online Encylopedia of
Integer Sequences (OEIS)~\cite{oeis}, and their main interpretation is the
number of permutations $\sigma$ of $n$ such that there are no indicies
$i<j<k$ satisfying
$\sigma(j+1)<\sigma(i)<\sigma(k)<\sigma(j)\quad\text{or}\quad\sigma(j)<\sigma(k)<\sigma(i)<\sigma(j+1).
$
They were introduced by  Baxter~\cite{Baxter64} in a
question about compositions of commuting functions, and have since been found in
many places in combinatorics.

Chung~\emph{et al.}~\cite{ChGrHoKl78} found the explicit formula for
the $n^\text{th}$ Baxter number given in Section~1, and also gave a
second order linear recurrence:
\begin{equation}\label{eqn:baxrec}
8(n+2)(n+1)B_n + (7n^2+49n+82) B_{n+1}-(n+6)(n+5) B_{n+2} =0.
\end{equation}

We prove Theorem~\ref{thm:Baxter} by a a generating function
argument. We use the notation $\overline{x}=\frac{1}{x}$, and work in
the ring of formal series $\mathbb{Q}[x, \overline{x}] [\![t]\!]$.
The operator $PT_x$ (resp $NT_x$) extracts positive (resp. negative)
powers of $x$ in series of $\mathbb{Q}[x, \overline{x}] [\![t]\!]$. We
relate this to the diagonal operator $\Delta$ defined by : $\Delta
\sum_{i,j \in \mathbb{N}} f_{i,j} x^it^j = \sum_n f_{n,n} t^n$.

A bivariate series in $\mathbb{Q}[x] [\![t]\!]$ is D-finite with
respect to $t$ and $x$ if the set of its partial derivatives spans a
finite dimensional vector space over
$\mathbb{Q}$. Lipshitz~\cite{Lips88} proved that if $F(x,t)\in
\mathbb{Q}[x] [\![t]\!]$ is D-finite then so is $f(t)=\Delta
F(x,t)$. This result is effective, and creative telescoping strategies
have resulted in efficient algorithms for computing such
differential equations~\cite{Chyz00, BoLaSa13}: that is, given a
system of differential equations satisfied by $F(x,t)$, one can
compute the linear differential equation satisfied by $\Delta
F(x,t)=f(t)$. There are several implementations in various
computer algebra systems, such as \textsf{Mgfun} by Chyzak in Maple~\cite{Chyz94}, and the \textsf{HolonomicFunctions}
package of Koutschan in Mathematica~\cite{Kout10}.

Roughly, the strategy is to use elimination in an Ore algebra or an ansatz of undetermined coefficients
to compute differential operators annihilating the multivariate integral
\begin{equation} \frac{1}{2\pi i} \int_\Omega \frac{F(x,t/x)}{x} dx = \Delta F(x,t), \end{equation}
where $\Omega$ is an appropriate contour in $\mathbb{C}$ containing the origin.

\subsection{A generating function for hesitating walks ending on an axis}
We recall Proposition~12 of Bousquet-M\'elou and Xin~\cite{BoXi05}; As
noted earlier, we have shifted our indices, and our presentation of
their results are duely adapted.  Here $Q_k$ is the first quadrant
$Q_k=\{(x_1, \dots, x_k): x_i> 0\}$.
\begin{proposition}[Bousquet-M\'elou and Xin, Propostion 12 of~\cite{BoXi05}]
  For any starting and ending points $\lambda$ and $\mu$ in $W_k$, the
  number of $W_k$-hesitating walks going from $\lambda$ to $\mu$ can
  be expressed in terms of the number of $Q_k$ hesitating walks as
  follows:
\[
w_k(\lambda, \mu, n) = \sum_{\pi\in S_k} (-1)^\pi q_k(\lambda, \pi(\mu), n),
\]
where $(-1)^\pi$ is the sign of $\pi$ and $\pi(\mu_1, \dots,
\mu_k)=(\mu_{\pi(1)}, \dots, \mu_{\pi(k)})$.
\end{proposition}
The result is proved by a classic reflection argument using a simple
sign reversing involution between pairs of walks; the walks restricted
to $W_k$ appear as fixed points.
We restrict to $k=2$, and define
\[
H(x;t) = \sum_{i,n} q_2((2,1), (i,1), 2n) x^it^n\quad \text{and}\quad 
V(y;t) = \sum_{i,n} q_2((2,1), (1,i), 2n) y^it^n.
\]
By applying the proposition we see immediatetly that
\begin{equation}
\label{eqn:hv}
W(x,t)=\sum_{i,n} w_2((2,1), (i,1), 2n)\, x^it^n = H(x;t)-V(x;t).
\end{equation}
Their explicit Proposition~13 is key to our solution. 
\begin{proposition}[Bousquset-M\'elou and Xin, Proposition~13
  of~\cite{BoXi05}]
\label{thm:BX}
  The series $H(x;t)$ and $V(y;t)$ which count $Q_2$-hesitating walks
  of even length ending on the $x$-axis and on the $y$-axis, satisfy
\begin{eqnarray}
xH(x;t) &= PT_x \frac{Y}{t(1+x)} (x^2-Y^2/x^2+Y/x^3)\\
\overline{x}^2V(\overline{x};t)&= NT_x \frac{Y}{t(1+x)} (x^2-Y^2/x^2+Y/x^3),
\end{eqnarray}
where $Y$ is the algebraic function
\[Y= \frac{-tx^2+(1-2t)x -t\sqrt{{t}^{2}
{x}^{4}-2\,t{x}^{3}+ \left( -2\,{t}^{2}-4\,t+1 \right) {x}^{2}-
2\,tx+{t}^{2}}}{2t\left( 1+x \right) }.\]
\end{proposition}%
Using this, we can express $W(x,t)$ as a diagonal.
Define
\begin{equation}
G(x,t)=\frac{Y}{t(1+x)}\left(x^2-Y^2/x^2+Y/x^3\right).
\end{equation}

\begin{lemma}
\label{thm:setup}
The generating function for the class of walks of $W_2$-hesitating
walks of even length end on the $x$-axis defined
\[
W(t)=\sum_{i,n} w_2 ((2,1), (i,1), n) t^n 
\]
satisfies
\begin{equation}
W(t) = \Delta \left( \left(G\left(\frac{1}{x},xt\right)-G(x,xt)\right)
  \frac{1}{1-x} \right).
\end{equation}
\end{lemma}
\begin{proof} We remark that $W(t)=H(1;t)-V(1;t)$ by
  Equation~\ref{eqn:hv}. We apply Proposition~\ref{thm:BX}, and
  then convert $PT$ and $NT$ operators to diagonal operations by
straightforward series manipulation:
\begin{equation}
H(1;t)=\left.PT_x G(x,t)\right|_{x=1} = \Delta \left(G\left(\frac{1}{x},xt\right) \frac{1}{1-x} \right),\qquad V(1;t)= \Delta \left(G(x,xt) \frac{1}{1-x}\right).
\end{equation}
%
\end{proof}

\begin{theorem}\label{thm:baxter}
 Let $B_n$ be the number of Baxter permutations of size
  $n$. Then $W(t)=\sum B_{n+1} t^n$.
\end{theorem}
\begin{proof}
  We prove this using effective computations for D-finite closure
  properties.  Using the Mathematica package of
  Koutschan~\cite{Kout10} we use our expression for $W(t)$ as a
  diagonal to calculate that it satisfies the differential equation
  $\mathcal{L} \cdot W(t) =0$, where $\mathcal{L}$ is the differential
  operator
\begin{align*}
\mathcal{L} := & t^4 \left( t+1 \right)  \left( 8\,t-1 \right) {{\it D_t}}^{5}+{t}^{
3} \left( -20+147\,t+176\,{t}^{2} \right) {{\it D_t}}^{4} +4\,{t}^{2}
 \left( -30+241\,t+304\,{t}^{2} \right){{\it D_t}}^{3}  \\
&+12\,t \left( -
20+191\,t+256\,{t}^{2} \right) {{\it D_t}}^{2}
+24\, \left( -5+72\,t+104
\,{t}^{2} \right) {\it D_t}+240+384\,t.
\end{align*}

Furthermore, the Equation~\ref{eqn:baxrec} implies that the
shifted generating function for the Baxter numbers written above also
satisfies this differential equation. As the solution space of the
differential operator $\mathcal{L}$ is a vector space of dimension 5;
to prove equality it is sufficient to show that the first five terms
of the generating functions are equal. We have verified that the
initial terms in the series agree for over 200 terms.
\end{proof}

\subsection{A new generating tree}
\label{sec:GeneratingTree}
A \emph{generating tree} for a combinatorial class expresses recursive
structure in a rooted plane tree with labelled nodes. The objects of
size~$n$ are each uniquely generated, and the set of objects of size $n$
comprise the $n^\text{th}$ level of the tree. They are useful for
enumeration, and for showing that two classes are in bijection. One
consequence of Theorem~\ref{thm:Baxter} is a new generating tree
construction for Baxter objects.

Several different formalisms exist for generating trees,
notably~\cite{BaBoDeFlGaGo02}. The central properties are as
follows. Every object $\gamma$ in a combinatorial class $\mathcal{C}$ is
assigned a label $\ell(\gamma)\in\mathbb{Z}^k$, for some fixed
$k$. There is a rewriting rule on these labels with the property that
if two nodes have the same label then the ordered list of labels of
their children is also the same.  We consider labels that are pairs of
positive integers, specified by~$\{\ell_\text{Root}: [i,j]
\rightarrow \operatorname{Succ}([i,j])\}$, where $\ell_\text{Root}$
is the label of the root.

Two generating trees for Baxter objects are known in the literature,
and one consequence of Theorem~\ref{thm:Main} is a third, using the
generating tree for $\Omega^{(k)}$ given by Burrill~\emph{et al.}~\cite{Buetal12}. This
tree differs from the other two already at the third level,
illustrating a very different decomposition of the objects.  For the
three different systems, we give the succession rules, and the first 5
levels of the tree (unlabelled) in Figure~\ref{fig:trees}.
\begin{figure}\center
\begin{subfigure}[b]{0.45\textwidth}
                \includegraphics[width=\textwidth, height=2cm]{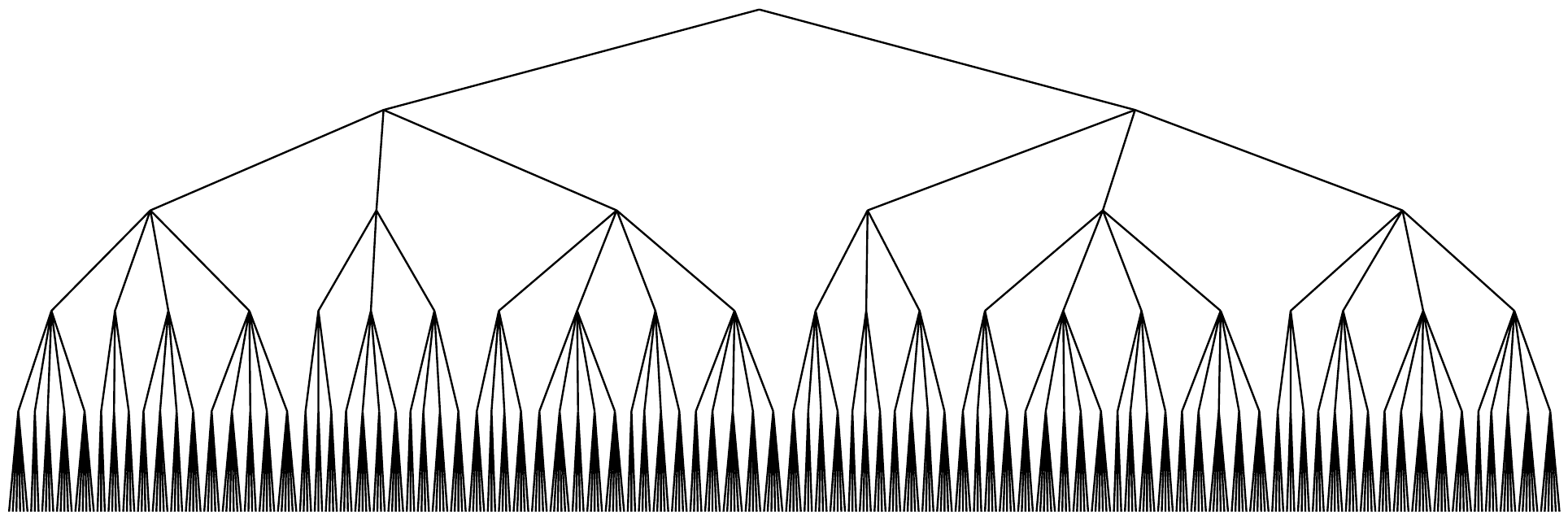}

{\tiny $\{[1,1];[i,j]\rightarrow [1,j+1], \dots, [i, j+1], [i+1, j], \dots [i+1, 1]\}$}
\end{subfigure}\begin{subfigure}[b]{0.45\textwidth}
                \includegraphics[width=\textwidth, height=2cm]{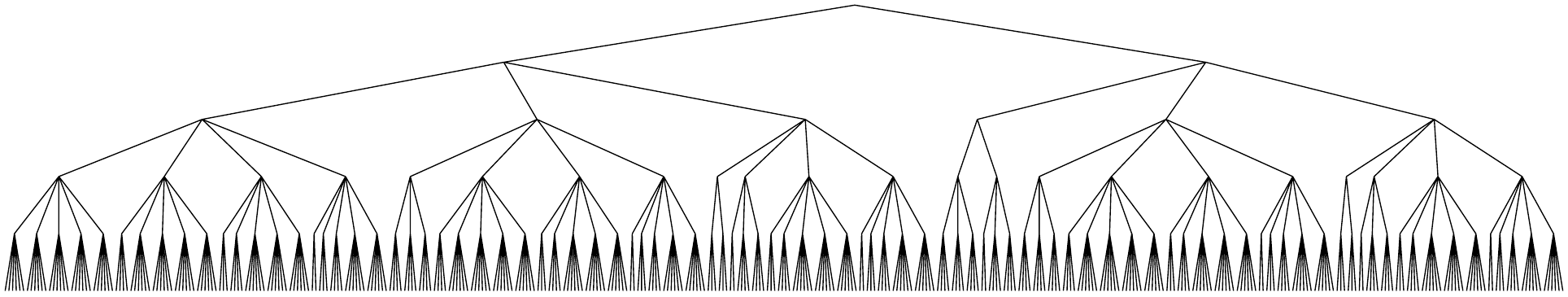}

{\tiny$\{[0,2];
[i,j]\rightarrow [0, j], \dots , [i-1, j], [1, j+1], \dots, [i+j-1, 2]\}$}
\end{subfigure}

\begin{subfigure}[b]{0.45\textwidth}
                \includegraphics[width=\textwidth,height=2cm]{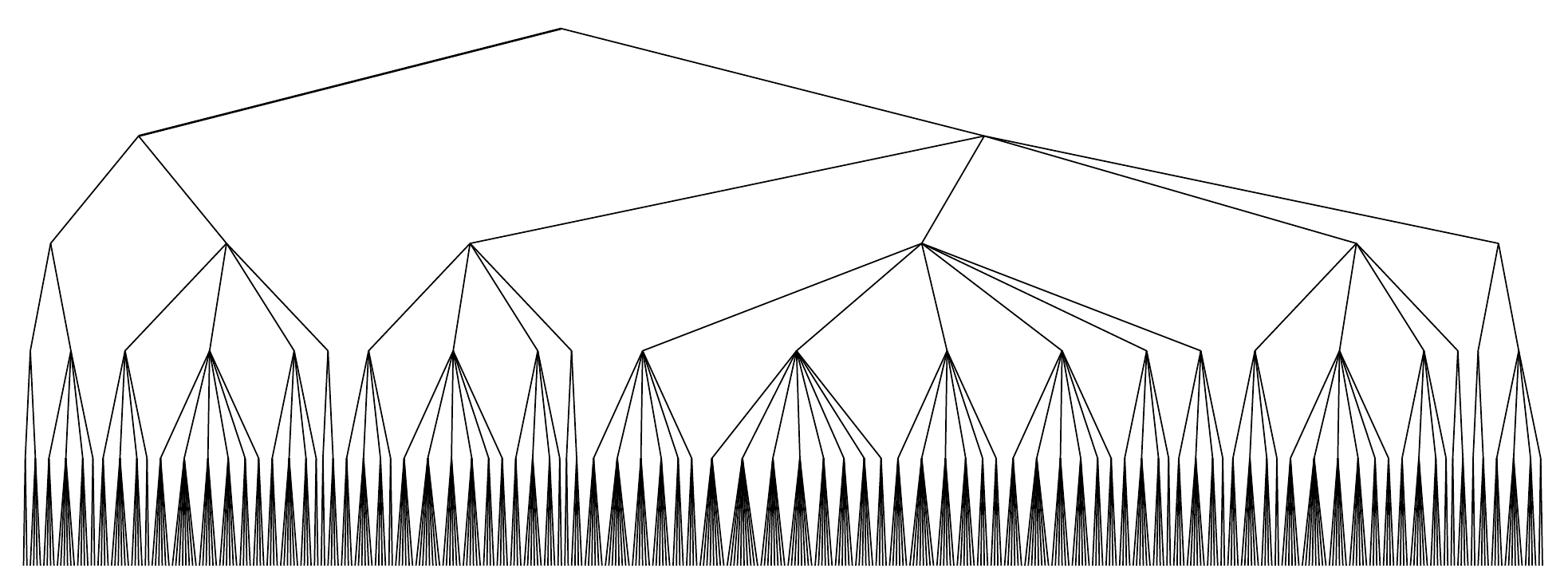}
{\tiny \[
\begin{array}{rll}
\{[0,0];[i,j]\rightarrow
&[i,i], [i +1,j]\\
&[i,j],[i,j+1],\dots,[i,i-1], &\mathrm{if}\, i>0\\
&[i-1,j],[i-1,j+1],\dots,[i-1,i-1], &\mathrm{if}\,i>0 \\ 
&[i , j -1], [i-1, j-1]&\mathrm{if}\, i > 0, \mathrm{and}\,  j > 0\}.\\
\end{array} \]}

\end{subfigure}

 \caption{The first five levels of each of the Baxter generating
   tree. They are respectively from~\cite{BoBoFu10}~\cite{BoGu14}~\cite{Buetal12}. }
\label{fig:trees}
\end{figure}

\section{Standard Young tableaux of bounded height}
\label{sec:Conjecture}
Chen \emph{et al.}~\cite{Chetal07} give an analogy comparing the
relationship between oscillating tableaux and irreducible
representations of the Brauer algebra to the relationship of standard
Young tableaux and the symmetric group. Indeed, there are many connections
between standard Young tableaux of bounded height and oscillating
tableaux of bounded height, but we were unable to find any
consideration of the following conjecture. 
\begin{conjecture}
  The set of oscillating lattice walks of length~$n$ in $W_k$ ending
  at the boundary $\{(m+k, k-1, \dots, 1 ): m\geq 0\}$ is in bijection
  with the set of standard Young tableaux of size~$n$, of height
  bounded by~$2k$.
\end{conjecture}
The case of $k=1$ is straightforward. An oscillating walk in~$W_1$ is
simply a sequence of $e_1$ and $-e_1$ steps such that at any point the
number of $e_1$ is greater than or equal to the number of
$-e_1$. There is a simple bijection to standard Young tableaux of
height 2: Given an oscillating walk as a sequence of steps,
$w=w_1,w_2,\dots, w_n$, the standard Young tableaux is filled by
putting entry~$j$ on the top row of the tableau if $w_j=e_1$, and on
the second row if $w_j=-e_1$. However, the interaction between the
steps is less straightforward in higher dimensions.

Eric Fusy remarked to us that the inverse RSK bijection maps standard
Young tableaux to involutions, as the images of pairs of identical
tableaux. This does not directly give our matching diagrams, but it is
a promising candidate for a bijection.

To give evidence for this conjecture, we first give expressions for
the exponential generating functions using determinants of matrices
filled with Bessel functions. For any $k$ the equivalence could
 be verified in this manner, and we have done so for $k\leq 8$. It is possible
that very explicit manipulations might also give an answer.

Next we express Young tableaux as walks, and use standard
generating function techniques for the ordinary generating
functions. This gives us a different characterization of the bijection
in terms of walks, and two expressions for the generating functions as
diagonals of rational functions. The technical details of the
constructions are reserved for the long version of this article. 

\subsection{A determinant approach}
Our first approach to settling this conjecture is a direct appeal
to two recent enumerative results. In this section, 
$b_j(x)=I_j(2x)=\sum_n\frac{(2x)^{2n+j}}{n!(n+j)!}$, the hyperbolic
Bessel function of the first kind of order~$j$.

Let the $\tilde{Y}_{k}(t)$ be the exponential generating function for
the class of standard Young tableaux with height bounded by
$k$. Formulas for $\tilde{Y}_k(t)$ follow from works of Gordon,
Houten, Bender and Knuth~\cite{Gord71, GoHo68, BeKn72}, which depend
on the parity of~$k$. We are only interested in the even values here.
\begin{theorem}[~\cite{Gord71, GoHo68, BeKn72}]
The exponential generating function for the class of standard Young
tableaux of height bounded by~$2k$ is given by
\[
\tilde{Y}_{2k}(t)= \det[b_{i-j}(t)+b_{i+j-1}(t)]_{1\leq i,j\leq k}.
\]
\end{theorem}

Around the same time, Grabiner-Magyar~\cite{GrMa93} determined an
exponential generating function for~$\mathcal{O}^{(k)}_n({\lambda}; {\mu})$,
the number of oscillating lattice walks of length~$n$
from~${\lambda}$ to~${\mu}$, which stay within the~$k$-dimensional Weyl
chamber~$W_k$ and take steps in the positive or negative unit coordinate
vectors.
\begin{theorem}[Grabiner-Magyar~\cite{GrMa93}] For fixed~${\lambda},
  {\mu} \in W_k$, the exponential
  generating function for $\mathcal{O}_n({\lambda}; {\mu})$ satisfies
\[
O_{{\lambda}, {\mu}}(t)= \sum_{n\geq 0} |\mathcal{O}_n({\lambda};
{\mu}) |\frac{t^n}{n!} 
= \det \left( b_{{\mu_i}-{\lambda_j}}(2t)-b_{{\mu_i}+{\lambda_j}}(2t)\right)_{1\leq i, j\leq k}.
\]
\end{theorem}
We specialize the start and end positions as
${\lambda}=\delta=(d, d-1, \ldots, 1)$ and~${\mu}=
\delta+me_1=(d+m', d-1, \ldots, 1)=(m, d-1, \ldots, 1)$. We are
interested in the sum over all values of $m$, and 
define  $O_k(t)\equiv\sum_{m\geq 0}O_{{\delta}, me_1+\delta}(t)$.
Using their result we deduce the following. 
\begin{proposition}
  The exponential generating function for the class of oscillating
  tableaux ending with a row shape is the finite sum
\[
\tilde{O}_k(t)= \sum_{u=0}^{k-1}(-1)^u \sum_{\ell=u}^{2k-1-2u}(I_{\ell}) \det (I_{i-j}-I_{kd-i-j})_{0 \leq i \leq k-1, i\neq u, 1\leq j \leq k-1}.
\]
\end{proposition}
This follows from the fact that the infinite sum which arises from direct
application of Grabiner and Magyar's formula telescopes, 
using also the identity $b_{-k}=b_k$.

Our conjecture is equivalent to
$\tilde{O}_k(t)=\tilde{Y}_{2k}(t)$. Here are the first two values:
\[
\tilde{O}_1(t)= \tilde{Y}_2(t)= b_0+b_1, \quad \tilde{O}_2(t) = \tilde{Y}_4(t)=b_0^2+b_0b_1+b_0b_3-2b_1b_2-b_2^2-b_1^2+b_1b_3.
\]
The determinants are easy to compute for small
values, and they agree for $k\leq 8$.

\subsection{A diagonal approach}
\label{sec:diag}
Using Theorem~\ref{thm:Main2}, we can reformulate the conjecture to be
strictly in terms of lattice paths by viewing standard Young tableaux as
oscillating tableaux with no deleting steps.
\begin{conjecture}
  The set of oscillating lattice walks of length $n$ in $W_k$ starting
  at $\delta=(k, k-1, \dots, 1)$ and ending at the boundary at the
  boundary $\{me_1+\delta: m\geq 0\}$ is in bijection with the set of
  oscillating lattice walks of length $n$ in $W_{2k}$, using only
  positive steps ($e_j$), starting at $\delta$ and ending anywhere in
  the region.
\end{conjecture}

To approach this, we consider the set of walks separately using
some standard enumeration techniques: namely the orbit sum method and
results on reflectable walks in Weyl Chambers. Some of the enumerative parallels
of these strategies in this context are discussed
in~\cite{MeMi14c}. The advantage of these diagonal representations is
potential access to asymptotic enumeration formulas, and possibly
alternative combinatorial representations. All of the generating
functions are D-finite, and we can use the work of~\cite{BoLaSa13}
to determine bounds on the shape of the annihilating differential
equation.
\begin{theorem}
\label{thm:osc}
The ordinary generating function for oscillating walks starting at $\delta$ and
ending on the boundary $\{me_1+\delta: m\geq 0\}$, is given by the
following formula:
\[ O_k(t)=\Delta \left[ \frac{t^{2d-1}(z_3 z_4^2 \cdots
    z_k^{k-2})(z_1+1)\prod_{1\leq j<i \leq k} (z_i-z_j)(z_iz_j-1)
    \cdot \prod_{2 \leq i \leq k} (z_i^2 -1)}{1-t(z_1\cdots z_k)(z_1 +
    \oz_1 + \cdots z_d + \oz_d)}  \right]. \]
\end{theorem}
The proof of Theorem~\ref{thm:osc} is a rather direct application of
Gessel and Zeilberger's formula for reflectable walks in Weyl
chambers.

By applying an orbit sum analysis, we derive the following
expression. There are still some points to verify in the computation,
so although we believe it, and have verified it up to $k=8$, this form
remains as a conjecture. 
\begin{conjecture}The ordinary generating function for standard Young
  tableau of height at most $k$ is 
\[ Y_k(t) = \Delta \left( \frac{(z_1\cdots z_{k-1})\Phi(\overline{\bz})}{(1-t(z_1\cdots z_{k-1})S(\bz))(1-z_1)\cdots(1-z_{k-1})} \right), \]
where 
\[ S(\bz) = \oz_1 + \oz_1z_2 + \cdots + \oz_{k-1}z_{k-2} + z_{k-1}, \]
and
\[ \Phi(\bz) = \frac{(z_1z_{k-1}-1)}{(z_1\cdots z_{k-1})^{k-1}}
\prod_{j=1}^{k-2}(z_1z_j - z_{j+1})\prod_{j=2}^{k-1}(z_{k-1}z_j -
z_{j-1})\prod_{j=1}^{k-3}\prod_{k=j+2}^{k-1}(z_jz_k-z_{j+1}z_{k-1}). \]
\end{conjecture}
In order to prove the conjecture, giving an explicit diagonal representation for the generating function of the number of 
standard Young tableaux of a given height, it is sufficient to prove that for every $d$
{\small 
\[ \sum_{\sigma \in \mathcal{G}} (-1)^{\sgn(g)} \sigma(z_1\cdots z_{d-1}) = \frac{(z_1z_{k-1}-1)}{(z_1\cdots z_{k-1})^{k-1}}
\prod_{j=1}^{k-2}(z_1z_j - z_{j+1})\prod_{j=2}^{k-1}(z_{k-1}z_j -
z_{j-1})\prod_{j=1}^{k-3}\prod_{k=j+2}^{k-1}(z_jz_k-z_{j+1}z_{k-1}),\] }
where $\mathcal{G}$ is the finite group of rational transformations of
$\mathbb{R}^{d-1}$ generated by{\small 
\begin{align*}
\phi_1: (z_1,\dots,z_{d-1}) &\mapsto (\oz_1z_2,\dots,z_{d-1}) \\
(1 < k < d-1) \quad \phi_k: (z_1,\dots,z_{d-1}) &\mapsto (z_1,\dots,z_{k-1},\quad z_{k-1}\oz_kz_{k+1}, \quad z_{k+1}, \dots,z_{d-1}) \\
\phi_{d-1}: (z_1,\dots,z_{d-1}) &\mapsto (z_1,\dots,z_{d-1},z_{d-2}\oz_{d-1}),
\end{align*}}
which acts on $f \in \mathbb{R}[z_1,\dots,z_d]$ by $\sigma(f(z_1,\dots,z_{d-1})) = f(\sigma(z_1,\dots,z_{d-1}))$.  It can be shown that $\mathcal{G}$ is isomorphic to the symmetric group $S_d$, however the action of its elements in terms such as $\sigma(z_1\cdots z_{d-1})$ is less clear.

\section{Conclusion}
\label{sec:Conclusion}
These problems are clearly begging for combinatorial proofs. There are
a huge number of candidates for Baxter objects to choose from, and
hopefully the Catalan subclasses are retained. 

The diagonal expressions in Section~\ref{sec:diag}, which are of
interest in their own right, are also compelling in a wider
context. We are quite interested in the conjecture of Christol on
whether or not globally bounded D-finite series can always be
expressed as a diagonal of a rational function. The class of standard
Young tableau of bounded height is a very intriguing case, and
follows our study of lattice path walks~\cite{MeMi14c} wherein we
consider the study of diagonal expressions for known D-finite
classes.
\section*{Acknowledgements}
We are extremely grateful to Eric Fusy, Julien Courtiel, Sylvie
Corteel, Lily Yen, Yvan le Borgne and Sergi Elizalde for stimulating
conversations, and important insights.

\small

\end{document}